\documentclass[11pt, a4paper]{amsart}
\usepackage{a4wide}
\usepackage{graphicx,enumitem}
\usepackage{amsmath,amssymb}
\usepackage[english]{babel}
\usepackage{subfigure}
\usepackage{units}
\usepackage{diagbox}
\usepackage{fontenc}
\usepackage{xcolor}
\usepackage{caption}
\usepackage{ulem}
\usepackage{cancel}
\usepackage{tikz}
\usepackage{amssymb}
\DeclareSymbolFontAlphabet{\amsmathbb}{AMSb}%
\usepackage{mathbbol}
\usepackage{latexsym}

\usepackage[pdftex,
pdftitle={SFEM approximation of spherical Gaussian random fields},
bookmarksopen,
colorlinks,
linkcolor=black,
urlcolor=black,
citecolor=black
]{hyperref}
\hypersetup{pdfauthor={E. Jansson, M. Kov\'acs, and A. Lang}}

\usepackage{dsfont}

\usepackage{mathtools}
\DeclarePairedDelimiter{\ceil}{\lceil}{\rceil}
\DeclarePairedDelimiter{\floor}{\lfloor}{\rfloor}

\newcommand{\IP}{\amsmathbb{P}}
\newcommand{\R}{\amsmathbb{R}}

\newcommand{\N}{\amsmathbb{N}}

\newcommand{\IS}{\amsmathbb{S}}

\newcommand{\gb}{\beta}

\renewcommand{\gg}{\gamma}
\newcommand{\gk}{\kappa}

\newcommand{\gO}{\Omega}

\newcommand{\cB}{\mathcal{B}}

\newcommand{\cF}{\mathcal{F}}

\newcommand{\cL}{\mathcal{L}}

\newcommand{\cW}{\mathcal{W}}

\newcommand{\fbet}{\{\beta\}}
\DeclareMathOperator{\E}{\amsmathbb{E}} 
\DeclareMathOperator{\Cov}{\mathsf{Cov}}

\newcommand{\dd}{\,\mathrm{d}}

\newcommand{\KL}{Karhunen--Lo\`eve }

\newtheorem{lemma}{Lemma}[section]
\newtheorem{proposition}[lemma]{Proposition}

\newtheorem{theorem}[lemma]{Theorem}
\theoremstyle{remark}

\theoremstyle{definition}

\usepackage{color}
\definecolor{darkgreen}{rgb}{0,.6,0}

\begin{document}
	\title[SFEM approximation of spherical Gaussian random fields]{Surface finite element approximation of spherical Whittle--Mat\'ern Gaussian random fields
	}

	\author[E.~Jansson]{Erik Jansson} \address[Erik Jansson]{\newline Department of Mathematical Sciences
	\newline Chalmers University of Technology \& University of Gothenburg
	\newline S--412 96 G\"oteborg, Sweden.} 
\email[]{erikjans@chalmers.se}
	
	\author[M.~Kov\'acs]{Mih\'aly Kov\'acs} \address[Mih\'aly Kov\'acs]{\newline Faculty of Information Technology and Bionics
		\newline P\'azm\'any P\'eter Catholic University
		\newline H-1444 Budapest, P.O. Box 278, Hungary.
		\newline and
		\newline Department of Mathematical Sciences
		\newline Chalmers University of Technology \& University of Gothenburg
		\newline S--412 96 G\"oteborg, Sweden.} \email[]{kovacs.mihaly@itk.ppke.hu}
	
	\author[A.~Lang]{Annika Lang} \address[Annika Lang]{\newline Department of Mathematical Sciences
		\newline Chalmers University of Technology \& University of Gothenburg
		\newline S--412 96 G\"oteborg, Sweden.} \email[]{annika.lang@chalmers.se}
		
	\thanks{
		Acknowledgement. 
		The author thank the anonymous referees for helpful comments.
		EJ and AL's work was partially supported by the Swedish Research Council (VR) through grant no.\ 2020-04170, by the Wallenberg AI, Autonomous Systems and Software Program (WASP) funded by the Knut and Alice Wallenberg Foundation, and by the Chalmers AI Research Centre (CHAIR). MK acknowledges the support of the Marsden Fund of the Royal Society of New Zealand through grant no.~18-UOO-143, the Swedish Research Council (VR) through grant no.\ 2017-04274, and the NKFIH through grant no.\ 131545}

	\subjclass{35R60, 60G60, 60H35, 58J05, 60H15, 65C30, 65N12, 65N30}

	\keywords{Stochastic partial differential equations. Gaussian random fields. Fractional operators. Parametric finite element methods. Strong convergence. Sphere. Surface finite element method.}
	
\begin{abstract}
Spherical Whittle--Mat\'ern Gaussian random fields are considered as solutions to fractional elliptic stochastic partial differential equations on the sphere. Approximation is done with surface finite elements. While the non-fractional part of the operator is solved by a recursive scheme, a quadrature of the Dunford--Taylor integral representation is employed for the fractional part. Strong error analysis is performed, and the computational complexity is bounded in terms of the accuracy. Numerical experiments for different choices of parameters confirm the theoretical findings. 		
\end{abstract}
	
	\maketitle

\section{Introduction}

In recent years Gaussian random fields (GRFs for short) have found use as a modeling tool in a variety of applications, such as geostatistics, materials science, and cosmology \cite{BARMAN2017,Guillot1999,Wandelt2012}. In many cases the domain of interest is~$\R^d$ or a subset thereof, but in some applications the scale of the domain makes it infeasible to disregard its geometry, for example in global geospatial modeling or simulation of the cosmic background radiation, see \cite{Lang2015,MP11, Porcu2018} and references therein. In these cases Gaussian random fields can instead be defined on the sphere making the study and simulation of these fields a topic of importance.

An example of a spherical Gaussian random field and our subject of study is the \emph{Whittle--Mat\'ern} field, which is defined as the solution~$u$ to the stochastic partial differential equation (SPDE) 
\begin{equation}\label{eq:whittle_euq}
(\kappa^2-\Delta_{\IS^2})^\beta u = \cW,
\end{equation}
where $\beta,\kappa>0$ are regularity parameters and $\cW$ denotes white noise on the sphere. Whittle--Mat\'ern fields are the spherical analogue to Mat\'ern fields on~$\R^d$~\cite{LRL11,W63}.  These random fields are of special interest since they are flexible in the sense that by only changing the two parameters $\beta$ and $\kappa$ one can obtain a wide range of smoothness and correlation lengths, where the former is determined by~$\beta$ and the latter by~$\kappa$~\cite{BK20,Guinness2016}. Therefore, they are often used in modeling which motivates the need for simulation methods for these particular fields. In this paper we propose a new simulation algorithm for any smoothness parameter $\gb>1/2$ based on surface finite elements and analyze its convergence and computational complexity. The advantage of the simulated random fields is their representation in terms of finite elements which makes them suitable as input noise to simulations of stochastic and random partial differential equations.

In the case of Euclidean domains, the fields are defined through their covariance functions which may serve as a starting point for simulations. On the sphere, however, simply substituting the great circle distance into the covariance function will not result in a valid covariance function~\cite{Guinness2016}. As earlier noted, Mat\'ern fields on surfaces are instead defined as solutions to the SPDE in Equation~\eqref{eq:whittle_euq}, which means that another approach is needed in the particular case of the sphere as well as in the general case of compact surfaces. 
	
One possible approach in the case of the sphere is to define a new family of admissible covariance functions that capture the desired covariance behavior~\cite{Alegria2021}. Another possibility is to use finite element techniques in order to approximate solutions to SPDE~\eqref{eq:whittle_euq}. Finite element approaches have been recently studied in the case of Euclidean domains, see for instance \cite{BK20,BKK18,BKK20,CK20}. 
For other recent papers considering simulation and sampling of Gaussian random fields using various methods, including fields on surfaces, see, e.g., \cite{AEL20, BD20,BK21, HHKS21,Herrmann2019, LP21+} and references therein.

If finite element methods are to be used in the spherical setting, a new challenge occurs compared to the Euclidean one, namely that a discretization of the geometry might be needed requiring an additional approximation. In this paper the framework used to discretize the geometry is the surface finite element method (SFEM) by Dziuk and Elliot~\cite{Dziuk2013}. Using SFEM in combination with a sinc quadrature approximation of the fractional part of the operator $(\kappa-\Delta_{\IS^2})^{-\beta}$ rewritten as a Dunford--Taylor integral, we manage to approximate solutions for all $\beta > 1/2$ by a recursive scheme with continuous finite elements without the need for higher order global smoothness.
In our main result Theorem~\ref{th:yt1}, we show convergence of $\operatorname{O}(h^{(2\gb-1)/(\gb+1)})$ with respect to the mesh size~$h$ when all error contributions are balanced. The computational work is bounded by $\operatorname{O}(h^{-3/2}(\ln h)^2)$ and could be further reduced to $\operatorname{O}(h^{-1}|\log h|^{7/2})$ using the preconditioning approach in~\cite{Herrmann2019}.

While spectral methods (see, e.g., \cite{CL18,Lang2015,LFR19} and references therein) and curved elements as in boundary element methods could be used in the specific case of this paper~\cite{Herrmann2019}, SFEM has, as a more traditional mesh-based approach, the advantage that it is easier to implement and compatible with existing software used in industry such as FEniCS~\cite{FEniCS} or DUNE~\cite{DUNE}. This paves the way to a broad application of the presented method in applications requiring the simulation of random fields on the sphere as input. Another benefit of the developed algorithm is the universality of the approach.
The setting of the particular operator $(\kappa^2-\Delta_{\IS^2})^\beta$ on the sphere serves as a stepping stone for development of more general operators on  a wider class of surfaces and manifolds.
It should be pointed out that while this method is presented with the main goal of simulating random fields in mind,  it is also possible to use it to solve non-random fractional elliptic partial differential equations  using low order finite elements. 
A natural extension of our approach using SFEM is to operators of the form $(\kappa^2-\nabla_{\IS^2}\cdot( A \nabla_{\IS^2}))$,  where $A \in L^\infty(\IS^2)$, which we leave as a topic for future work.

Furthermore, the algorithm can be extended to higher dimensions provided that the right hand side is sufficiently smooth, as is the case with truncated white noise expansions. This restriction arises due to the need to use Sobolev inequalities in the surface finite element error estimates \cite{Dziuk1988}, \cite[Remark 4.10]{Dziuk2013}. We emphasize that the study of random fields on two-dimensional surfaces is of special interest due to the relevance in applications.

The paper is structured as follows: In Section~\ref{seq:two} background material on the theory of random fields and functional analysis is introduced. This is used to derive a spectral representation of the solution to~\eqref{eq:whittle_euq} in terms of the spherical harmonic functions and a first convergence result for a spectral approximation. The section is concluded with the introduction of the recursive approach to~\eqref{eq:whittle_euq} that allows to approximate the solution with continuous finite elements without additional global smoothness assumptions. In Section~\ref{seq:three} the approximation of the fractional part of the operator is described, and convergence of the quadrature to the spectral approximation is shown. The surface finite element method is introduced in Section~\ref{seq:four} and the SFEM error is bounded. The full error analysis is presented in our main result Theorem~\ref{th:yt1}. A discussion on balancing the errors and estimating the computational work concludes the section. Finally, in Section~\ref{seq:five}, we give numerical experiments in FEniCS that confirm the theoretical findings.

\section{Isotropic Gaussian random fields on the sphere}
\label{seq:two}
We introduce basic properties of isotropic Gaussian random fields and their connection to solutions of stochastic partial differential equations in this section. The presentation is based on~\cite{Lang2015} and we refer the reader to~\cite{MP11} and~\cite{Lang2015} for more details. Convergence of a spectral approximation that will be used in later sections is also given.

The \emph{sphere~$\IS^2$} is defined by
\begin{equation*}
\IS^2=\left\{x \in \R^3 ~: \|x\|=1 \right\},
\end{equation*}
where $\|\cdot\|$ denotes the Euclidean norm and throughout this paper, $(\cdot,\cdot)$ refers to the corresponding inner product. We use the \emph{geodesic distance}, or great-circle distance, given by 
\begin{equation*}
d(x,y)=\arccos((x,y))
\end{equation*}
for $x,y \in \IS^2$ and denote by $\cB(\IS^2)$ the Borel $\sigma$-algebra on~$\IS^2$. The Lebesgue measure~$\dd A$ on the sphere is given by $\dd A=\sin(\theta)\dd \theta \dd\varphi$ with respect to spherical coordinates $\theta \in [0,\pi]$ and $\varphi \in [0,2\pi)$.

Let $L^2(\IS^2)$ denote the Hilbert space of square integrable functions. 		
The \emph{Laplace--Beltrami} operator on~$\IS^2$ is denoted by~$\Delta_{\IS^2}$. We define \emph{Sobolev spaces} with smoothness index $s \in \R^+$ via Bessel potentials by 	

\begin{equation*}
H^s(\IS^2)=\left(I -\Delta_{\IS^2}\right)^{-s/2}L^2(\IS^2).
\end{equation*}
The corresponding norm is given by 
\begin{equation*}
\|f\|_{H^s(\IS^2)}=\left\|\left(I-\Delta_{\IS^2}\right)^{s/2} f \right \|_{L^2(\IS^2)},
\end{equation*}
and for $s<0$, we define $H^s(\IS^2)$, as the space of distributions generated by
\begin{equation*}
H^s(\IS^2)= \left\{u=\left(I-\Delta_{\IS^2}\right)^{k}v, ~ v \in H^{2k+s}(\IS^2) \right\},
\end{equation*}
where $k \in \N$ is the smallest integer such that $2k+s>0$. In this case, the norm is given by 
\begin{equation*}
\|u\|_{H^s(\IS^2)} = \|v\|_{H^{2k+s}(\IS^2)}.
\end{equation*}
We set $H^0(\IS^2)=L^2(\IS^2)$. 
The reader is referred to~\cite{HLS18} and references therein for more details on Sobolev spaces defined using Bessel potentials. 

It is well known that the \emph{spherical harmonic functions}, denoted by $(Y_{l,m}, l \in \N_0, m=-l,\ldots,l)$, form an orthonormal basis for~$L^2(\IS^2)$ and that they are the eigenfunctions of the Laplace--Beltrami operator~$\Delta_{\IS^2}$. The corresponding eigenvalues are given by
\begin{equation*}
\Delta_{\IS^2} Y_{l,m} = - l(l+1) Y_{l,m}.
\end{equation*}

Let $(\Omega, \cF, \IP)$ be a complete probability space.
Similarly to~\cite{Lang2015}, we introduce a \emph{random field}~$Z$ on~$\IS^2$ as a $\cF \otimes \cB(\IS^2)$-measurable mapping $\gO \times \IS^2 \rightarrow \R$. The field is said to be \emph{isotropic} if the covariance function~$C$ only depends on the distance~$d$. In addition, the field is \emph{Gaussian} if it satisfies that $(Z(x_1),\ldots, Z(x_k))$ is multivariate Gaussian for any $k \in \N$ and $(x_1,\ldots,x_k) \in (\IS^2)^{ k}$. Without loss of generality, we assume that all considered fields are centered, i.e., $\E[Z] = 0$.

The field~$Z$ admits a basis expansion known as \emph{Karhunen--Lo\`eve expansion} with respect to the spherical harmonic functions
\begin{equation*}
Z(x) = \sum_{l=0}^\infty \sum_{m=-l}^l a_{l,m} Y_{l,m}(x).
\end{equation*}
Here $a_{l,m}= \int_{\IS^2}Z(y)\overline{Y_{l,m}}(y) \, \dd A(y)$ and the series expansion converges in $L^2(\Omega \times \IS^2;\R)$ and $L^2(\Omega;\R)$ for all $x \in \IS^2$.

Furthermore, there exists a sequence $(A_l, l\in \N_0)$ of nonnegative real numbers, known as the \emph{angular power spectrum}, such that for all pairs $l_1,l_2 \in \N_0$ and $m_i=-l_i,\ldots,l_i$, $i=1,2$, 
\begin{equation*}
\E[a_{l_1,m_1}\overline{a_{l_2,m_2}}]=A_{l_1} \delta_{l_1,l_2} \delta_{m_1 ,m_2},
\end{equation*}
where $\delta_{x,y}=1$ if $x=y$ and zero otherwise.
The random variables $a_{l,m}$ and $a_{l,-m}$ satisfy $a_{l,m}=(-1)^{m} \overline{a_{l,-m}}$ for $l \in \N$ and $m=1,\ldots,l$. 

Of importance in our SPDEs is the notion of spherical Gaussian white noise which is not a random field in~$L^2(\IS^2)$ but a so-called \emph{generalized random field} taking values in a larger space. More specifically, a \emph{Gaussian white noise}~$\cW$ on~$\IS^2$ is a centered Gaussian random field satisfying for any test functions $\phi, \psi \in L^2(\IS^2)$,
\begin{equation*}
\Cov\left((\cW,\phi )_{L^2(\IS^2)},(\cW,\psi )_{L^2(\IS^2)}\right)=(\phi,\psi)_{L^2(\IS^2)}.
\end{equation*} 	
Note that formally	$a_{l,m}=(\mathcal{W},Y_{l,m})_{L^2(\IS^2)}$ with 
\begin{equation*}
\Cov(a_{l_1,m_1},a_{l_2,m_2})
=(Y_{l_j,m_i},Y_{l_i,m_i})_{L^2(\IS^2)}=\delta_{l_j,l_i}\delta_{m_j,m_i}.
\end{equation*}
In other words we obtain
\begin{equation*}
A_{\mathcal{W},l} = \E[a_{l,m} \overline{a_{l,m}}]=1 
\end{equation*} 
and can as such formally view white noise as the field with angular power spectrum $A_{\mathcal{W},l}=1$ for all $l \in \N_0$, not converging in $L^2(\Omega;L^2(\IS^2))$.

Let us in what follows consider the class of isotropic Gaussian random fields generated by solutions to the fractional elliptic SPDE suggested in~\cite{LRL11}
\begin{equation}
\label{eq:theone}
\cL^{\beta} u= \left(\kappa^2-\Delta_{\IS^2}\right)^{\beta} u = \cW,
\end{equation}
where $\beta > 1/2$, $\kappa>0$, and $\cW$ denotes Gaussian white noise on the sphere.

Note that the solution~$u$ is an isotropic GRF satisfying
\begin{equation*}
u 
= \cL^{-\gb} \cW
= \sum_{l=0}^\infty \sum_{m=-l}^l a_{l,m} \cL^{-\gb} Y_{l,m}
= \sum_{l=0}^\infty \sum_{m=-l}^l a_{l,m} (\kappa^2+l(l+1))^{-\beta} Y_{l,m}
\end{equation*}
with angular power spectrum given by 
\begin{equation*}
A_l
=(\kappa^2+l(l+1))^{-2\beta}.
\end{equation*}
Since $\gb > 1/2$, the \KL expansion of $u$ converges in $L^2(\Omega;L^2(\IS^2))$ and the covariance operator is of trace-class with
\begin{align*}
\|u\|_{L^2(\gO;L^2(\IS^2))}^2
& = \E [ \|u\|_{L^2(\IS^2)}^2]
= \sum_{l=0}^\infty (2l+1) A_l
\le \int_0^\infty \frac{2x+1}{(\gk^2 + x(x+1))^{2\gb}} \, \dd x\\
& = \frac{\gk^{2(1-2\gb)}}{2\gb-1}
< + \infty.
\end{align*}

To give the reader an idea of the resulting random fields, we include two samples with respect to the same noise but different smoothness parameter~$\gb$ in Figure~\ref{fig:fields}.
\begin{figure}[tb]
	\centering
	\subfigure[$\gb = 0.75$ and $\gk = 1$.]{\includegraphics[width = 0.35\textwidth,trim={11cm 6.45cm 11cm 6.45cm},clip]{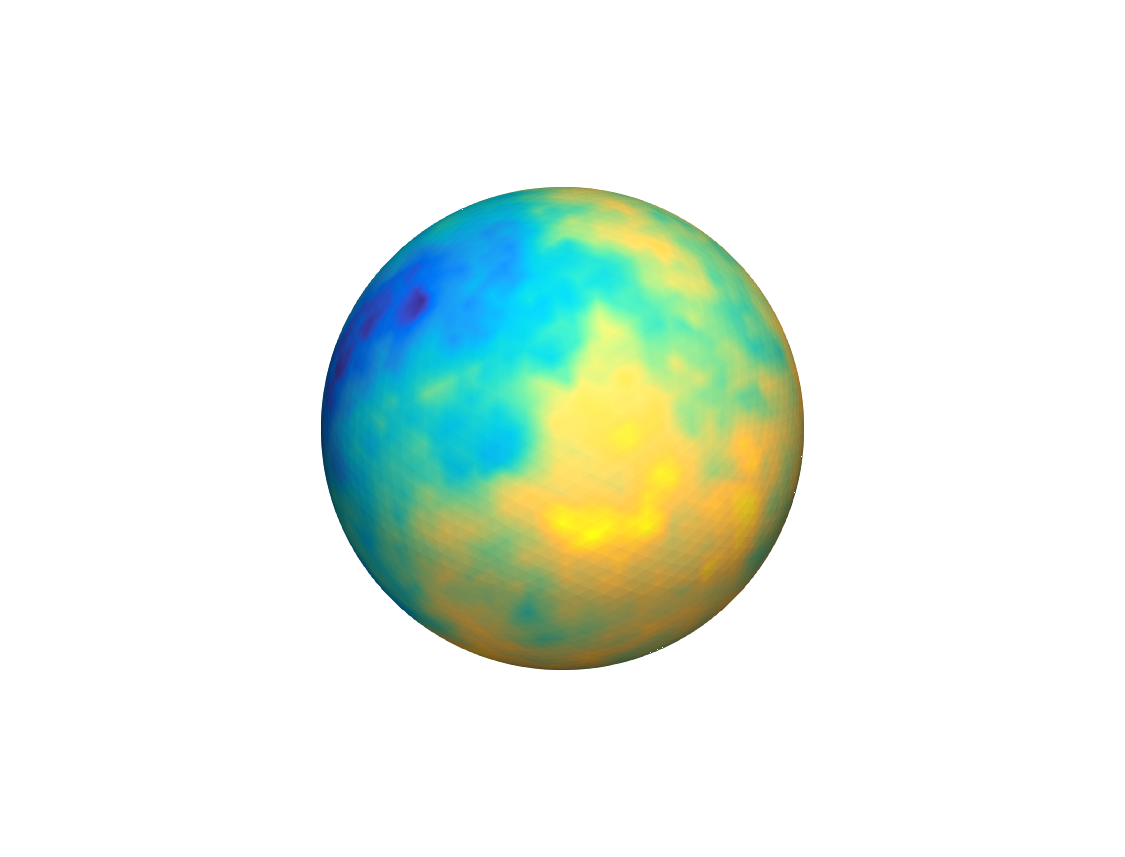}}
	\hspace*{5em}
	\subfigure[$\gb = 1.5$ and $\gk = 1$.]{\includegraphics[width = 0.35\textwidth,trim={11cm 6.45cm 11cm 6.45cm},clip]{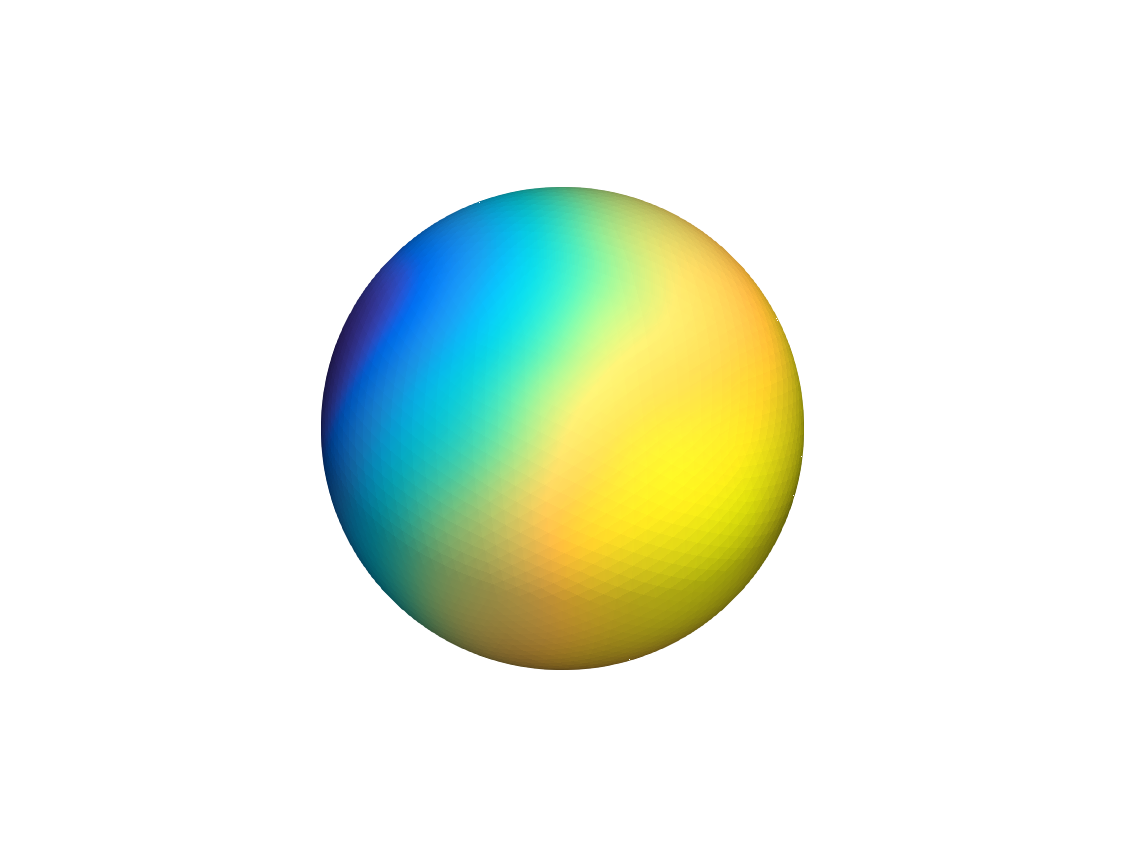}}
	\caption{Two Gaussian random field samples of solutions to~\eqref{eq:theone} generated using SFEM with the same noise but different values of the exponent~$\gb$. Here, the white noise expansion is truncated at $L=100$ and $h=0.051$.}
	\label{fig:fields}
\end{figure}
In order to obtain a finite-dimensional problem that is suitable for simulations and the approximation methods used in the following sections, let us consider the truncated white noise
\begin{equation*}
\cW_L 
=\sum_{l=0}^L \sum_{m=-l}^l a_{l,m} Y_{l,m} 
\end{equation*}
with $A_l = 1$ for all $l \le L$ and $A_l = 0$ otherwise, which satisfies that
\begin{equation}\label{eq:cov_W_L}
\|\cW_L\|_{L^2(\gO;L^2(\IS^2))}^2
= \sum_{l=0}^L (2l+1)
= (L+1)^2.
\end{equation}
The corresponding SPDE with smooth right hand side becomes
\begin{equation}\label{eq:smooth_SPDE}
\cL^{\beta} u_L
= \cW_L,
\end{equation}
where the solution~$u_L$ is an isotropic GRF with
\begin{equation}\label{eq:KL_uL}
u_L 
= \sum_{l=0}^L \sum_{m=-l}^l a_{l,m} (\kappa^2+l(l+1))^{-\beta} Y_{l,m}
\end{equation}
and
\begin{align}\label{eq:Cov_uL}
\begin{split}
\|u_L\|_{L^2(\gO;L^2(\IS^2))}^2
&= \sum_{l=0}^L (2l+1) A_l\le \int_0^L \frac{2x+1}{(\gk^2 + x(x+1))^{2\gb}} \, \dd x\\
&\le \frac{\gk^{2(1-2\gb)} - (\gk^2 + L(L+1))^{1-2\gb}}{2\gb-1}.
\end{split}
\end{align}

As a direct consequence of Proposition~5.2 in~\cite{Lang2015} we obtain the following result.
\begin{proposition}\label{prop:truncref}
	Let $u$ and $u_L$ be the solutions to~\eqref{eq:theone} and~\eqref{eq:smooth_SPDE}, respectively. Then there exists $C_\gk > 0$ such that for any $L \ge 1$,
	\begin{equation*}
	\|u-u_L\|_{L^2(\Omega;L^2(\IS^2))} 
	\leq C_{\gk}\left(\frac{1}{2\beta-1}+\frac{1}{4\beta-1}\right)L^{1-2\beta}.
	\end{equation*}
\end{proposition}

Having obtained a first spectral approximation and its speed of convergence, we continue with rewriting \eqref{eq:smooth_SPDE} for $\gb > 1$ as a system of SPDEs suitable for finite element methods.

For $\gb > 1$ let $\floor{\gb}$ denote the integer part of~$\gb$ and $\{\beta\}=\beta-\floor{\beta}$ its fractional, i.e., $\floor{\beta} \in \N_0$ and $\{\beta\}\in [0,1)$. For $\{\gb\} \neq 0$, we rewrite~\eqref{eq:smooth_SPDE} as a system of equations given by the recursion
\begin{equation}\label{eq:system_SPDE}
\cL u_L^i = u_L^{i-1}
\end{equation}
for $i=1,\ldots, \floor{\gb}$ with $u_L^0 = \cW_L$ and
\begin{equation}\label{eq:system_fractional_SPDE}
\cL^{\{\gb \}} u_L = u_L^{\floor{\gb}}.
\end{equation}
For $\{\gb \} = 0$, we are in the non-fractional setting and set $u_L = u_L^{\floor{\gb}}$.

We observe that
\begin{equation*}
u_L^i = \cL^{-i} \cW_L
\end{equation*}
and therefore by~\eqref{eq:KL_uL} and~\eqref{eq:Cov_uL} for $i \ge 1$,
\begin{align}\label{eq:Cov_uLi}
\begin{split}
\|u_L^i\|_{L^2(\Omega;L^2(\IS^2))} 
&\leq \left(\frac{\kappa^{2(1-2i)}-(\kappa^2+L(L+1))^{1-2i}}{2i-1}\right)^{1/2}\\
&\leq \kappa^{1-2i} (2i-1)^{-1/2} 
< + \infty.
\end{split}
\end{align}

The recursion scheme allows us to approximate solutions to the fractional problem. First, we can use SFEM for~$\cL$ to approximate solutions to the first $\floor{\beta}$ non-fractional SPDEs recursively. We emphasize that this is an advantage compared to approximating $\cL^{\floor{\gb}}$ directly since higher order operators would require higher order conforming finite element spaces. In the final step, we approximate the fractional operator in such a way that even the solution to the last problem in the recursion can be approximated using SFEM.

\section{Approximation of fractional operators}
\label{seq:three}
In order to develop a finite element approximation of~\eqref{eq:theone}, we approximate the fractional operator in the last step of the recursion~\eqref{eq:system_fractional_SPDE} by a quadrature. 
By~\cite[Theorem 2.1]{Bonito2013}, we can write the inverse of the fractional operator~$\cL^{\{\gb\}}$ as a Dunford--Taylor integral
\begin{equation}
\label{eq:bala2}
\mathcal{L}^{-\{\beta\}} =\frac{\sin( \pi \fbet)}{\pi}\int_{-\infty}^{\infty} e^{2 \fbet y} \left( I+e^{2y}\mathcal{L}\right)^{-1} \, \dd y.
\end{equation}
We partition the range of~$y$ into an equidistant grid with step size~$k$, and following~\cite{Bonito2013} approximate the integral in~\eqref{eq:bala2} using a sinc quadrature, thus obtaining 
\begin{equation*}
\mathcal{L}^{-\fbet} \approx Q^{\fbet}_k= \frac{2 k \sin(\pi \fbet)}{\pi} \sum_{l=-K^{-}}^{K^+} e^{2\fbet y_{l}} \left(I+e^{2 y_l}\mathcal{L}\right)^{-1}
\end{equation*}
with $y_{l}=kl$. Furthermore,
\begin{equation*}
K^+=\ceil*{\frac{\pi^2}{4(1-\fbet)k^2}},
\qquad
K^{-}=\ceil*{\frac{\pi^2}{4\fbet k^2}},
\end{equation*}
where $\ceil*{\cdot}$ denotes rounding up to the closest integer.

We approximate the solution to Equation~\eqref{eq:system_fractional_SPDE} by
\begin{equation} \label{eq:diff}
u_{L,Q,k}= Q^{\fbet}_k u_L^{\floor{\beta}}=\frac{2 k \sin(\pi \fbet)}{\pi} \sum_{l=-K^{-}}^{K^+} e^{2\fbet y_{l}} \left(I+e^{2 y_l}\mathcal{L}\right)^{-1} u_L^{\floor{\beta}},
\end{equation}
where the expressions $(I+e^{2 y_l}\mathcal{L})^{-1} u_L^{\floor{\beta}}$ on the right hand side are obtained by solving the subproblems
\begin{equation}\label{eq:SPDE_subproblems}
u_l+e^{2y_l}\mathcal{L}u_l
=\left(1+e^{2y_l}\kappa^2\right)u_l-e^{2y_l}\Delta_{\IS^2} u_l
=u_L^{\floor{\beta}}.
\end{equation}

We bound the error between $u_L=\cL^{-\fbet} u^{\floor{\beta}}_L$ and $u_{L,Q,k}$ by employing the analysis of the exponentially convergent sinc quadrature approximation of~\eqref{eq:bala2} developed in~\cite{Bonito2013}. The following proposition is an application of \cite[Theorem~3.5]{Bonito2013} to the setting of this paper.

\begin{proposition}
	Let $\gb >1/2$ with $\{\gb\} \neq 0$. Further, let $u_L$ be given by~\eqref{eq:system_fractional_SPDE} and $u_{L,Q,k}$ by~\eqref{eq:diff}. The error is then bounded for any finite $L>0$ by 
	\begin{align*}
	&\|u_L-u_{L,Q,k}\|_{L^2(\Omega;L^2(\IS^2))}\\
	&\qquad \leq \frac{2 \sin(\pi\fbet)}{\pi}\left(\frac{1}{2\fbet}+\frac{1}{\kappa^2(2-2\fbet)}\right)\left(\frac{e^{-\pi^2/(4k)}}{\sinh(\pi^2/(4k))}+e^{-\pi^2/(2k)}\right)\\
	& \hspace*{10em} 
	\times \|u^{\floor{\beta}}_L\|_{L^2(\Omega;L^2(\IS^2))}\\
	& \qquad \le \frac{2 \sin(\pi\fbet)}{\pi}\left(\frac{1}{2\fbet}+\frac{1}{\kappa^2(2-2\fbet)}\right)\left(\frac{e^{-\pi^2/(4k)}}{\sinh(\pi^2/(4k))}+e^{-\pi^2/(2k)}\right)\\
	& \hspace*{10em} 
	\times \left( \delta_{0 ,\floor{\beta}}(L+1) + (1-\delta_{0 ,\floor{\beta}}) \frac{\gk^{1-2\floor{\beta}}}{(2\floor{\beta}-1)^{1/2}} \right)\\
	& \qquad = c_1(k, L, \gb),
	\end{align*}
	where the right hand side $c_1(k, L, \gb)$ is exponentially decaying in~$k$.
	\label{prop:quaderr}
\end{proposition}

We remark that the theorem as given in~\cite{Bonito2013} is also valid for $\gb <1/2$ which is not of relevance in the context of this paper.

Since the proposition follows by first noting that the largest eigenvalue of~$\cL^{-1}$ is given by~$\kappa^{-2}$ and then applying the definition of the $L^2(\Omega;L^2(\IS^2))$ norm to the estimate in~\cite[Theorem 3.5]{Bonito2013}, we omit the proof. We note that the finite-dimensional setting of~\cite{Bonito2013} applies since the truncated \KL series of~$\cW$ leads to an SPDE on the finite-dimensional subspace of~$L^2(\IS^2)$ spanned by the spherical harmonics of the first $L$~eigenvalues of~$\cL$.

We observe that in simulations with a coarse mesh size~$k$ and a small correlation length parameter~$\kappa$, the constant $c_1(k,L,\gb)$ can become very large even though it decays exponentially as $k \to 0$. This is due to the fact that the smallest eigenvalue of the operator~$\cL$ goes to zero as $\kappa \to 0$. This problem can be remedied by refining the quadrature with a smaller~$k$.

\section{SFEM approximation and its strong convergence}
\label{seq:four}
Having approximated the noise and the fractional operator in the previous sections, it remains to approximate solutions to the linear subproblems~\eqref{eq:system_SPDE} and~\eqref{eq:SPDE_subproblems} appearing in the recursion and sinc quadrature.

The weak formulation of~\eqref{eq:system_SPDE} is given by:
Find $u_L^i \in H^1(\IS^2)$ such that 
\begin{equation}\label{eq:weak_regular_SPDE}
\mathfrak{a}_{\IS^2}(u_L^i,v)=(u_L^{i-1},v)_{L^2(\IS^2)}
\end{equation}
for every $v \in H^1(\IS^2)$, where the bilinear form $\mathfrak{a}_{\IS^2}:H^1(\IS^2) \times H^1(\IS^2) \to \R$, is given by 
\begin{equation}\label{eq:bilinear_form}
\mathfrak{a}_{\IS^2}(u,v)=\kappa^2\left(u,v\right)_{L^2(\IS^2)} + \left( \nabla_{\IS^2} u ,\nabla_{\IS^2}v \right)_{L^2(\IS^2)}.
\end{equation}

This bilinear form is obtained by integration by parts, where this particular expression is obtained due to the compactness of the sphere \cite[Theorem 2.10, Theorem 2.14]{Dziuk2013}. Note furthermore that the bilinear form is coercive and continuous, thus implying the existence of solutions by virtue of the Lax--Milgram theorem.

Likewise, the weak formulation of~\eqref{eq:SPDE_subproblems} is given by:
Find $u \in H^1(\IS^2)$ such that 
\begin{equation} \label{eq:weak_SPDE}
\mathfrak{a}_{\IS^2,l}(u,v)=( u^{\floor{\beta}}_L,v)_{L^2(\IS^2)}
\end{equation}
for every $v \in H^1(\IS^2)$,
where the bilinear form $\mathfrak{a}_{\IS^2,l}: H^1(\IS^2) \times H^1(\IS^2) \to \R$, is given by 
\begin{equation}
\label{eq:subprobbiform}
\mathfrak{a}_{\IS^2,l}(u,v)=\left(1+e^{2y_l}\kappa^2 \right) \left( u, v\right)_{L^2(\IS^2)} +e^{2y_l} \left( \nabla_{\IS^2} u ,\nabla_{\IS^2}v \right)_{L^2(\IS^2)}.
\end{equation}

We approximate the solutions to these problems by using the surface finite element method of~\cite{Dziuk2013}. In what follows we describe SFEM in the particular case of the sphere for the completeness of our presentation.

By $\IS_h^2$ we denote an approximation of~$\IS^2$ with a piecewise polygonal surface consisting of non-degenerate triangles with vertices on~$\IS^2$, where $h$ refers to the size of the largest triangle, which is defined as the in-ball radius.

For two triangles $T$ and~$\widetilde{T}$, it holds that either $\widetilde{T} \cap T=\emptyset$ or that their intersection is their common edge or vertex. Let us denote by $\mathcal{T}_h$ the set of triangles making up the discretized sphere~$\IS^2_h$, i.e.,
\begin{equation*}
\IS^2_h=\bigcup_{T_j \in \mathcal{T}_h} T_j.
\end{equation*}
To give an impression of the resulting geometry, we visualize one discretized sphere and a possible refinement in Figure~\ref{fig:discr_sphere}.
\begin{figure}[tb]
	\subfigure[Discretized sphere.]{\includegraphics[width=0.45\textwidth]{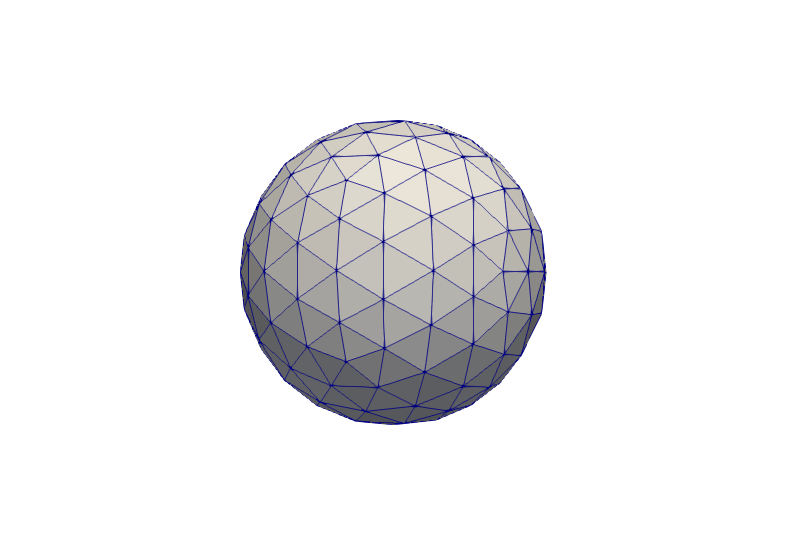}}
	\subfigure[Refinement of the sphere.]{\includegraphics[width=0.45\textwidth]{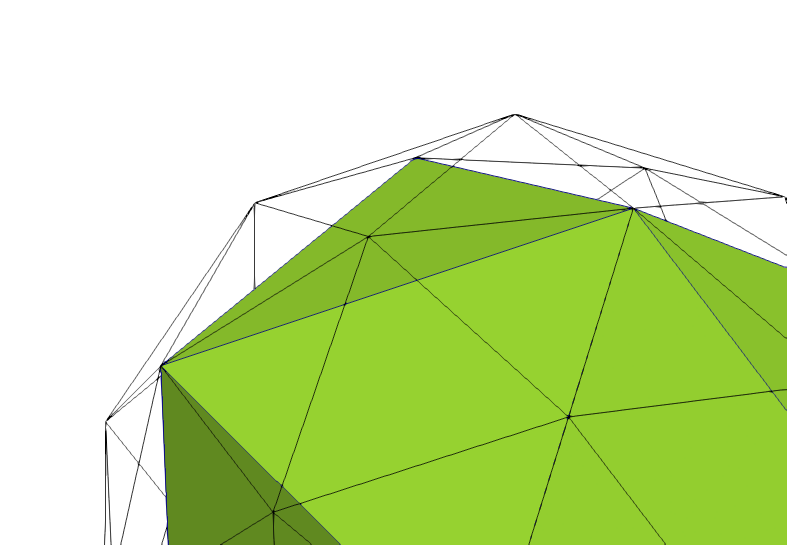}}
	\caption{Discretized polygonal approximation of the sphere.\label{fig:discr_sphere}}
\end{figure}

The \emph{signed distance function} $d_s: \R^3 \to \R$ to~$\IS^2$ is given by 
\begin{equation*}
d_s(x)=\|x\|-1
\end{equation*}
for $x$ both outside and inside of the sphere. As such, it can take both negative and positive values, warranting the name \emph{signed} distance function.

By~\cite{Dziuk2013}, $d_s$ is smooth and for $U= \left\{ x \in \R^3: \left|d_s(x)\right|< 1\right\} \supset \IS^2$, the projection $p:U \rightarrow\IS^2$ given by
\begin{equation*}
p(x)=x-d_s(x) \nu(x) 
\end{equation*} 
is onto, where $\nu$ denotes the outward normal on~$\IS^2$. Restricted to $\IS_h^2 \subset U$, $p: \IS_h^2 \rightarrow \IS^2$ becomes an isomorphism. Therefore, a function $\eta: \IS_h^2 \rightarrow \R$ may be lifted to~$\IS^2$ by setting
\begin{equation*}
\eta^\ell = \eta \circ p^{-1},
\end{equation*}
where we emphasize that $\ell$ is used as abbreviation for the \emph{lift} and should not be understood as a parameter.

For every $T \in \mathcal{T}_h$, we define a lifted triangle $T^\ell \subset \IS^2$ by $T^\ell=p(T)$. The procedure is illustrated in Figure~\ref{fig:lift} in the one-dimensional setting. Note that the points on the discretized surface are \emph{lifted} along the normal of the surface. This pointwise evaluation allows us to define $\eta^\ell$, since $\eta$ is evaluated on its original domain.
\begin{figure}[htb]
	\centering
	\begin{tikzpicture}[x=0.75pt,y=0.75pt,yscale=-1,xscale=1]
	
	\draw  [draw opacity=0] (164.75,124.77) .. controls (177.87,68.56) and (241.13,25.76) .. (317.37,25.44) .. controls (395.98,25.11) and (461.16,70.05) .. (471.95,128.74) -- (317.88,146.19) -- cycle ; \draw   (164.75,124.77) .. controls (177.87,68.56) and (241.13,25.76) .. (317.37,25.44) .. controls (395.98,25.11) and (461.16,70.05) .. (471.95,128.74) ;
	\draw  [dash pattern={on 4.5pt off 4.5pt}]  (164.75,124.77) -- (153.5,175) ;
	\draw  [dash pattern={on 4.5pt off 4.5pt}]  (471.95,128.74) -- (484.5,182) ;
	\draw    (166.73,123.51) -- (318.52,26.27) ;
	\draw [shift={(320.5,25)}, rotate = 327.36] [color={rgb, 255:red, 0; green, 0; blue, 0 }  ][line width=0.75]      (0, 0) circle [x radius= 3.35, y radius= 3.35]   ;
	\draw [shift={(164.75,124.77)}, rotate = 327.36] [color={rgb, 255:red, 0; green, 0; blue, 0 }  ][line width=0.75]      (0, 0) circle [x radius= 3.35, y radius= 3.35]   ;
	\draw    (322.44,26.33) -- (351.39,46.16) -- (470.01,127.41) ;
	\draw [shift={(471.95,128.74)}, rotate = 34.41] [color={rgb, 255:red, 0; green, 0; blue, 0 }  ][line width=0.75]      (0, 0) circle [x radius= 3.35, y radius= 3.35]   ;
	\draw [shift={(320.5,25)}, rotate = 34.41] [color={rgb, 255:red, 0; green, 0; blue, 0 }  ][line width=0.75]      (0, 0) circle [x radius= 3.35, y radius= 3.35]   ;
	\draw    (463.5,102) -- (487.71,89.89) ;
	\draw [shift={(489.5,89)}, rotate = 513.4300000000001] [color={rgb, 255:red, 0; green, 0; blue, 0 }  ][line width=0.75]    (10.93,-3.29) .. controls (6.95,-1.4) and (3.31,-0.3) .. (0,0) .. controls (3.31,0.3) and (6.95,1.4) .. (10.93,3.29)   ;
	\draw    (396.22,76.87) -- (412.29,55.6) ;
	\draw [shift={(413.5,54)}, rotate = 487.07] [color={rgb, 255:red, 0; green, 0; blue, 0 }  ][line width=0.75]    (10.93,-3.29) .. controls (6.95,-1.4) and (3.31,-0.3) .. (0,0) .. controls (3.31,0.3) and (6.95,1.4) .. (10.93,3.29)   ;
	\draw  [dash pattern={on 0.84pt off 2.51pt}]  (444.5,112) -- (466.12,99.57) ;
	\draw [shift={(455.31,105.79)}, rotate = 510.11] [color={rgb, 255:red, 0; green, 0; blue, 0 }  ][line width=0.75]    (6.56,-1.97) .. controls (4.17,-0.84) and (1.99,-0.18) .. (0,0) .. controls (1.99,0.18) and (4.17,0.84) .. (6.56,1.97)   ;
	\draw  [fill={rgb, 255:red, 0; green, 0; blue, 0 }  ,fill opacity=1 ] (441.5,112) .. controls (441.5,110.34) and (442.84,109) .. (444.5,109) .. controls (446.16,109) and (447.5,110.34) .. (447.5,112) .. controls (447.5,113.66) and (446.16,115) .. (444.5,115) .. controls (442.84,115) and (441.5,113.66) .. (441.5,112) -- cycle ;
	\draw  [fill={rgb, 255:red, 0; green, 0; blue, 0 }  ,fill opacity=1 ] (461.22,98.85) .. controls (459.67,100.01) and (459.35,102.2) .. (460.5,103.75) .. controls (461.65,105.3) and (463.85,105.62) .. (465.4,104.47) .. controls (466.95,103.32) and (467.27,101.12) .. (466.12,99.57) .. controls (464.96,98.02) and (462.77,97.7) .. (461.22,98.85) -- cycle ;
	
	\draw (465,68.4) node [anchor=north west][inner sep=0.75pt]    {$\nu $};
	\draw (380,44.4) node [anchor=north west][inner sep=0.75pt]    {$\nu _{h}$};
	\draw (422,113.4) node [anchor=north west][inner sep=0.75pt]  [font=\scriptsize]  {$x$};
	\draw (480,104.4) node [anchor=north west][inner sep=0.75pt]  [font=\scriptsize]  {$p( x)$};
	\draw (229,85.4) node [anchor=north west][inner sep=0.75pt]    {$\IS^{1}_{h}$};
	\draw (369,7.4) node [anchor=north west][inner sep=0.75pt]    {$\IS^{1}$};
	\end{tikzpicture}
	\caption{One dimensional illustration of the lift.}
	\label{fig:lift}
\end{figure}
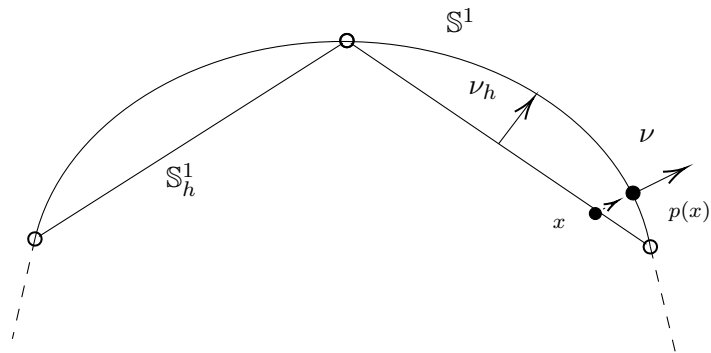

In order to be able to discretize problems defined on~$\IS^2_h$, define the finite element space
\begin{equation*}
S_h=\left\{\phi_h \in C^0(\IS^2_h): \phi_h|_T \in \mathcal{P}^1(T), T \in \mathcal{T}_h\right\} \subset H^1(\IS^2_h),
\end{equation*}
where $\mathcal{P}^1(T)$ denotes the space of all polynomials of degree at most one.
The lifted finite element space is given by 
\begin{equation*}
S_h^\ell=\left\{\varphi_h =\phi_h^\ell: \phi_h \in S_h \right\} \subset H^1(\IS^2). 
\end{equation*}

The \emph{tangential gradient} $\nabla_{\IS^2_h} \eta$ of a function $\eta: \IS^2_h \to \R$ is defined in a pointwise sense by 
\begin{equation*}
\nabla_{\IS^2_h} \eta(x)=\left(\delta_{i,j}-\nu_{h,i}\nu_{h,j}\right)(I-d_s(x)H_s(x))\nabla_{\IS^2}\eta^{\ell}(p(x)),
\end{equation*}
where $\nu_{h,i}$ denotes the outward normal of the $i$-th triangle~$T_i$
and $H_s$ is the Hessian of the signed distance function~$d_s$, which is given by 
\begin{align*}
H_s(x)= \frac{1}{\|x\|} \begin{bmatrix}1-\frac{x_1^2}{\sqrt{\|x\|}} &-\frac{x_1x_2}{\sqrt{\|x\|}} &-\frac{x_1x_3}{\sqrt{\|x\|}} \\ -\frac{x_1x_2}{\sqrt{\|x\|}} &1-\frac{x_2^2}{\sqrt{\|x\|}} &-\frac{x_2x_3}{\sqrt{\|x\|}} \\ -\frac{x_1x_3}{\sqrt{\|x\|}} &-\frac{x_2x_3}{\sqrt{\|x\|}} &1-\frac{x_3^2}{\sqrt{\|x\|}}\end{bmatrix}.
\end{align*}

Given this short introduction to SFEM on the sphere, we are now ready to formulate the discretized problems used in the recursion. Since they are linear elliptic SPDEs, the method in~\cite{Dziuk2013}
can be used. We define the bilinear forms on~$\IS_h^2$ corresponding to~\eqref{eq:bilinear_form} and~\eqref{eq:subprobbiform} by 
\begin{equation*}
\mathfrak{a}_{\IS^2_h}(u,v)
=\kappa^2(u,v)_{L^2(\IS^2_h)}
+ \left( \nabla_{\IS^2_h} u ,\nabla_{\IS^2_h}v \right)_{L^2(\IS^2_h)}
\end{equation*}
and
\begin{equation*}
\mathfrak{a}_{\IS^2_h,l}(u,v)
=\left(1+e^{2y_l}\kappa^2 \right) \left( u, v\right)_{L^2(\IS^2_h)} 
+e^{2y_l} \left( \nabla_{\IS^2_h} u ,\nabla_{\IS^2_h}v \right)_{L^2(\IS^2_h)}
\end{equation*}
for $u,v \in H^1(\IS^2_h)$, respectively.

The weak formulations of~\eqref{eq:weak_regular_SPDE} and~\eqref{eq:weak_SPDE} on the discretized sphere are hence given by:
Find $\tilde{u}_{L,h}^{i} \in H^1(\IS^2_h)$ such that 
\begin{equation}
\mathfrak{a}_{\IS_h^2}(\tilde{u}_{L,h}^{i},v)=(\tilde{u}^{i-1}_{L,h},v)_{L^2(\IS_h^2)}
\label{eq:linproblem}
\end{equation}
for all $v \in H^1(\IS^2_h)$. 
And similarly:
Find $\tilde{u}_{L,h} \in H^1(\IS^2_h)$ such that 
\begin{equation}
\mathfrak{a}_{\IS_h^2,l}(\tilde{u}_{L,h},v)=(\tilde{u}^{\floor{\beta}}_{L,h},v )_{L^2(\IS_h^2)}
\label{eq:subproblem}
\end{equation}
for all $v \in H^1(\IS^2_h)$. 

Here, $\tilde{u}_{L,h}^0 =\cW_{L,h}$ will denote an approximation of the white noise on the discretized sphere.
One way to obtain this is to lift an approximation of the truncated white noise~$\cW_L$ on~$\IS_h^2$ to the sphere.

There are different methods to obtain $\cW_{L,h}$ and its corresponding lift~$\cW_{L,h}^\ell$. One possibility is to use interpolation as done in~\cite[Lemma 4.3]{Dziuk2013}. To this end, let $u$ be any function in~$H^2(\IS^2)$.
Denote the $N$ nodes of $\IS^2_h$ by $(x_1,\ldots,x_N)$. For every $T \in \mathcal{T}_h$, it holds that the nodes lie on~$\IS^2$. We construct $\tilde{I_h} u \in S_h \subset H^1(\IS^2_h)$ by first setting
\begin{equation*}
\tilde{I}_h u (x_i)=u(x_i),
\end{equation*}
and then performing linear interpolation using the basis functions of~$S_h$. 
Define $I_h: H^2(\IS^2) \to S_h^\ell \subset H^1(\IS^2)$ by lifting the interpolated function $\tilde{I}_h u \in S_h$ to~$\IS^2$, that is to say, $\cW_{L,h}= \tilde{I}_h \cW_L$ and $\cW_{L,h}^\ell = I_h \cW_L$. By adapting 
\cite[Lemma 4.3]{Dziuk2013}, it is straightforward to show that 
\begin{equation}\label{eq:W_Lh_bound}
\|\cW_L- \cW_{L,h}^\ell \|_{{L^2(\gO;L^2(\IS^2))}} 
\leq \sqrt{2} ch^2 \|\cW_L\|_{{L^2(\gO;H^2(\IS^2))}}.
\end{equation}
We observe that 
\begin{align}\label{eq:W_L_H2_bound}
\begin{split}
\|\cW_L\|_{L^2(\Omega;H^2(\IS^2))}^2
& =\E\left[\|(I-\Delta_{\IS^2})\cW_L\|_{L^2(\IS^2)}^2\right]\\
& = \sum_{l=0}^L (1+l(l+1))^2(2l+1)
\le C (L+1)^6,
\end{split}
\end{align}
where the last bound follows from Faulhaber's formula.
This yields
\begin{equation*}
\|\cW_L- \cW_{L,h}^\ell \|_{{L^2(\gO;L^2(\IS^2))}} 
\leq Ch^2 (L+1)^3
\end{equation*}
for some constant~$C$.

One of the perks of the interpolation approach is that we manage to deal with the geometric error stemming from the discretization of~$\IS^2$, but a drawback is that the factor $\|\cW_L\|_{L^2(\gO;H^2(\IS^2))}$ will grow cubically in~$L$ due to the high regularity assumptions.

Another way to obtain $\cW_{L,h}^\ell$ is to use an orthogonal projection of $\cW_L$ onto $S_h^\ell$. 
It is done by finding $P_h \cW_L \in S_h^\ell$ such that $(\cW_L-P_h \cW_L, v)_{L^2(\IS^2)}=0$ for all $v \in S_h^\ell$.
This equation yields a system of equations for the coefficients of the lift of the nodal basis of~$S_h^\ell$. By solving this system, we obtain $P_h \cW_L =\cW_{L,h}^\ell$ Since for any $v \in S_h^\ell$, it holds that $(\cW_L-P_h \cW_L,v)_{L^2(\IS^2)}=0$,
\begin{align*}
\|\cW_L-P_h \cW_L\|_{L^2(\IS^2)}^2
&= (\cW_L-P_h \cW_L,\cW_L-v)_{L^2(\IS^2)}\\
&\leq \|\cW_L-P_h\cW_L\|_{L^2(\IS^2)}\|\cW_L-v\|_{L^2(\IS^2)}.
\end{align*}
What remains to do is to choose $v \in S_h^\ell$. If we choose $v=0$, we have that 
\begin{equation*}
\|\cW_L-P_h \cW_L\|_{L^2(\Omega;L^2(\IS^2)} \leq \|\cW_L\|_{L^2(\Omega;L^2(\IS^2))}.
\end{equation*}
If we instead let $v=I_h \cW_L$, we obtain from~\eqref{eq:W_Lh_bound} that 
\begin{equation}
\label{eq:projectionH2}
\|\cW_L-P_h \cW_L\|_{L^2(\Omega;L^2(\IS^2))} 
\leq \sqrt{2} c h^2\|\cW_L\|_{L^2(\Omega;H^2(\IS^2))}.
\end{equation}
By interpolation we therefore obtain for $s \in [0,2]$
\begin{equation}
\label{eq:inters}
\|\cW_L-P_h \cW_L\|_{L^2(\Omega;L^2(\IS^2))} 
\leq (\sqrt{2} c h)^s\|\cW_L\|_{L^2(\Omega;H^s(\IS^2))}
\leq C h^s (L+1)^{s+1},
\end{equation}
where the last inequality is obtained similarly to~\eqref{eq:W_L_H2_bound} with $2$ substituted by~$s$.
We note that as usual, in order to obtain convergence in~$h$, higher order norms of~$\cW_L$ have to be bounded which grow faster in~$L$ the higher the order of the Sobolev space.

Let us return to the weak formulations~\eqref{eq:linproblem} and~\eqref{eq:subproblem} and introduce their SFEM approximations:
Find $u_{L,h}^{i} \in S_h$ such that 
\begin{equation}
\mathfrak{a}_{\IS_h^2}(u_{L,h}^{i},v_h)=(u^{i-1}_{L,h},v_h)_{L^2(\IS_h^2) }
\label{eq:linproblemSFEM}
\end{equation}
for all $v_h \in S_h$.
And similarly:
Find $u_{h,l} \in S_h$ such that
\begin{equation}
\label{eq:subprobeey}
\mathfrak{a}_{\IS_h^2,l}(u_{h,l},v_h)= (u^{\floor{\gb}}_{L,h} ,v_h)_{L^2(\IS_h^2) }
\end{equation}
for all $v_h \in S_h$ and $u^{\floor{\gb}}_{L,h} \in S_h$ obtained in~\eqref{eq:linproblemSFEM}. 

In order to bound the error of the two approximations~\eqref{eq:linproblemSFEM} and~\eqref{eq:subprobeey} in a common setting, we observe that the bilinear forms~\eqref{eq:bilinear_form} and~\eqref{eq:subprobbiform} only differ by their coefficients. Therefore, we consider a general continuous and coercive bilinear form of the form
\begin{equation*}
b_{\IS^2}(u,v)
= A(u,v)_{L^2(\IS^2)}+B \left( \nabla_{\IS^2} u ,\nabla_{\IS^2}v \right)_{L^2(\IS^2)}
\end{equation*}
with coefficients~$A,B \in \R$ (that may be chosen such that we obtain $\mathfrak{a}_{\IS^2}$ or $\mathfrak{a}_{\IS^2,l}$) and its corresponding bilinear form on~$\IS_h^2$
\begin{equation*}
b_{\IS_h^2}(u,v)
= A(u,v)_{L^2(\IS^2_h)}
+ B \left( \nabla_{\IS^2_h} u ,\nabla_{\IS^2_h}v \right)_{L^2(\IS^2_h)}.
\end{equation*}
We then consider the problems: Given $U\in L^2(\IS^2)$ find $u\in H^1(\IS^2)$ such that
\begin{equation}\label{eq:weak_general_SPDE}
b_{\IS^2}(u,v) = (U,v)_{L^2(\IS^2)}
\end{equation}
for $v\in H^1(\IS^2)$ and: Given $U_h\in S_h$ find $u_h\in S_h$ such that
\begin{equation}\label{eq:SFEM_general_SPDE}
b_{\IS_h^2}(u_h,v) = (U_h,v)_{L^2(\IS_h^2)},
\end{equation}
for all $v_h\in S_h$. We will then choose $U$ and $U_h$ as in~\eqref{eq:weak_regular_SPDE}, \eqref{eq:weak_SPDE}, \eqref{eq:linproblemSFEM}, and~\eqref{eq:subprobeey}, respectively. 
\begin{proposition}
	\label{prop:SFEM_error}
	Let $u$ be the weak solution to~\eqref{eq:weak_general_SPDE} with $U \in L^2(\gO;L^2(\IS^2))$ being a general right hand side, and denote by $u_h^\ell$ the lifted solution to~\eqref{eq:SFEM_general_SPDE} with $U_h^\ell \in L^2(\gO;L^2(\IS^2))$.
	Then the strong error is bounded by	
	\begin{equation*}
	\|u - u_{h}^\ell\|_{L^2(\Omega;L^2(\IS^2))}
	\leq c \gamma^2 \left(
	h^2 \|U\|_{L^2(\Omega;L^2(\IS^2))}
	+ \|U-U_h^\ell\|_{L^2(\Omega;L^2(\IS^2))}
	\right),
	\end{equation*}
	where $\gg = \max(A,B)$.
	If in addition $U_h^\ell$ converges to~$U$ in $L^2(\gO;L^2(\IS^2))$ then $u_{h}^\ell$ converges to~$u$.
\end{proposition}

\begin{proof}	
		The claim follows from the corresponding deterministic inequality
		\begin{equation*}
		\|u-u_{h}^\ell\|_{L^2(\IS^2)} \leq c\gamma^2 \left( h^2 \|U\|_{L^2(\IS^2)}+\|U-U_h^\ell\|_{L^2(\IS^2)}\right)
		\end{equation*}
		with non-random general right hand side $U$ and~$U_h^\ell$, respectively. It is proven by traditional finite element techniques and an Aubin--Nitsche duality argument. The bilinear forms defined on~$\IS^2$ and~$\IS_h^2$ are compared using estimates of the geometric errors of the bilinear forms, see \cite[Lemma 4.7]{Dziuk2013}. For details of the proof, see \cite[Section~4]{Jansson2019} as well as~\cite{Dziuk1988,Dziuk2013}.
\end{proof}

Before stating and proving a bound on the error of the entire recursion scheme, we begin by stating and proving a proposition which allows us to bound the error of the final fractional problem. We introduce our final approximation (for $\fbet \neq 0$) as
\begin{equation}\label{eq:finalapprox}
u_{L,h}^\ell
=\frac{2 k \sin(\pi \fbet)}{\pi} \sum_{l=-K^{-}}^{K^+} e^{2\fbet y_{l}}u_{h,l}^{\ell},
\end{equation}
where $u_{h,l}^{\ell}$ is the lifted solution to~\eqref{eq:subprobeey}.

\begin{proposition}\label{lem:firststep}
	Let $u_{L,Q,k}$ be given by~\eqref{eq:diff} and $u_{L,h}^\ell$ be given by~\eqref{eq:finalapprox}. Then 
	\begin{equation*}
	\|u_{L,Q,k}-u_{L,h}^\ell\|_{L^2(\Omega; L^2(\IS^2))}
	\leq c_2(k,\beta) \left( (L+2)h^2 + \|\cW_L-\cW_{L,h}^\ell\|_{L^2(\Omega;L^2(\IS^2))}\right)
	\end{equation*}
	for some constant $c_2(k,\beta)$.
\end{proposition}

\begin{proof}
	We note that we can bound $\|u_{L,Q,k}-u_{L,h}^\ell\|_{L^2(\Omega, L^2(\IS^2))}$ by using Equations~\eqref{eq:diff}, \eqref{eq:SPDE_subproblems}, and~\eqref{eq:subprobeey} together with the triangle inequality,
	\begin{equation*}
	\|u_{L,Q,k}-u_{L,h}^\ell\|_{L^2(\Omega; L^2(\IS^2))}
	\leq \frac{2 k |\sin(\pi \{\beta)\}|}{\pi}\sum_{l=-K^{-}}^{K^+} e^{2\{\beta\} y_{l}}
	\left\| u_{l}-u_{h,l}^\ell\right\|_{L^2(\Omega;L^2(\IS^2))}.
	\end{equation*}
	
	Let us start by bounding $\| u_{l}-u_{h,l}^\ell\|_{L^2(\Omega;L^2(\IS^2))}$.
	Applying Proposition~\ref{prop:SFEM_error} with $\gamma_l=\max(1+e^{2y_l}\kappa^2,e^{2y_l})$ and right hand side $U=u^{\floor{\beta}}_L $ and $U_h=u_{L,h}^{\floor{\beta}}$ yields
	\begin{equation*}
	\| u_{l}-u_{h,l}^\ell\|_{L^2(\Omega;L^2(\IS^2))}
	\le c \gg_l^2 \left(h^2 \|u^{\floor{\beta}}_L\|_{L^2(\Omega, L^2(\IS^2))} 
	+ \|u_L^{\floor{\beta}}-u_{L,h}^{\floor{\beta},\ell}\|_{L^2(\Omega;L^2(\IS^2))}\right),
	\end{equation*}
	where $u_{L,h}^{\floor{\beta},\ell}$ denotes the lifted solution to~\eqref{eq:linproblemSFEM} or the lifted white noise approximation if $\floor{\beta}= 0$.
	By Proposition~\ref{prop:SFEM_error}, we can recursively bound the last term by
	\begin{align*}
	\|u_L^{\floor{\beta}}-u_{L,h}^{\floor{\beta},\ell}\|_{L^2(\Omega;L^2(\IS^2))}
	& \le h^2 \sum_{i=0}^{\floor{\beta}-1} (c\gamma^2)^{\floor{\beta}-i} \|u^i_L\|_{L^2(\Omega, L^2(\IS^2))}\\
	&\qquad +(c\gamma^2)^{\floor{\beta}}\|\cW_L-\cW_{L,h}^\ell\|_{L^2(\Omega;L^2(\IS^2))}
	\end{align*}
	with $\gg = \max\{1,\gk^2\}$.
	The estimates~\eqref{eq:cov_W_L} and~\eqref{eq:Cov_uLi} yield
	\begin{equation*}
	\|u^i_L\|_{L^2(\Omega, L^2(\IS^2))}
	\le
	\begin{cases}
	L+1, & i = 0,\\
	\kappa^{1-2i}(2i-1)^{-1/2}, & \text{else},
	\end{cases}
	\end{equation*}
	which leads to
	\begin{align*}
	\|u_L^{\floor{\beta}}-u_{L,h}^{\floor{\beta},\ell}\|_{L^2(\Omega;L^2(\IS^2))}
	& \le h^2 \sum_{i=1}^{\floor{\beta}-1} c^{\floor{\beta}-i} \max \{ \gk^{1-2i}, \gk^{4\floor{\beta}+1-6i} \} (2i-1)^{-1/2}\\
	&\quad~ {+}c^{\floor{\beta}} \max\{1, \gk^{4\floor{\gb}} \} \!
	\!\left(\! h^2 (L{+}1) {+} \|\cW_L{-}\cW_{L,h}^\ell\|_{L^2(\Omega;L^2(\IS^2))} \!\right)\!.
	\end{align*}
	This implies that the overall error is bounded by
	\begin{align*}
	& \| u_{l}-u_{h,l}^\ell\|_{L^2(\Omega;L^2(\IS^2))}\\
	&\quad~ \le c \gg_l^2 h^2\! \left(\!
	\sum_{i=1}^{\floor{\beta}} c^{\floor{\beta}-i}\!\max\{ \gk^{1-2i}, \gk^{4\floor{\beta}{+}1-6i} \} (2i{-}1)^{-1/2}
	{+}c^{\floor{\beta}} \max\{1, \gk^{4\floor{\gb}} \}\!(L{+}1)\!\right)\! \\
	&\quad~\qquad +c \gg_l^2 c^{\floor{\beta}} \max\{1, \gk^{4\floor{\gb}} \} \|\cW_L-\cW_{L,h}^\ell\|_{L^2(\Omega;L^2(\IS^2))}\\
	&\quad~ = c \gg_l^2 \left(
	\left(c_{\gk,\floor{\gb}}^{(1)} + c_{\gk,\floor{\gb}}^{(2)} \, (L+1) \right) h^2 
	+ c_{\gk,\floor{\gb}}^{(2)} \, \|\cW_L-\cW_{L,h}^\ell\|_{L^2(\Omega;L^2(\IS^2))}\right),
	\end{align*}
	and we observe that $\gg_l$ is the only component that depends on~$l$. Therefore, it only remains to estimate
	\begin{align*}
	\sum_{l=-K^{-}}^{K^+} e^{2\{\beta\} y_{l}} \gg_l^2
	& \leq 4 \max \{1,\gk^4\} \sum_{l=-K^{-}}^{K^+} e^{2(\fbet+2) k l }\\
	& = 4 \max \{1,\gk^4\}
	\!\left(\sum_{l=1}^{K^{-}} e^{-2(\fbet+2) k l }{+}\sum_{l=0}^{K^+}e^{2(\fbet+2) k l }\right)\!\\
	& = 4 \max \{1,\gk^4\}
	\!\left(\!
	e^{-2(\fbet+2)k} \frac{1{-}e^{-2(\fbet+2)k K^-}}{1{-}e^{-2(\fbet+2)k} }
	{+}\frac{e^{2(\fbet+2)k(K^++1)}{-}1}{e^{2(\fbet+2)k}{-}1} \!\right)\!\\
	& = c_{\gk, \{\gb\}, k},
	\end{align*}
	where we used the properties of the geometric series in the last step.
	This allows us to finally obtain
	\begin{equation*}
	\|u_{L,Q,k}-u_h^\ell\|_{L^2(\Omega; L^2(\IS^2))}
	\leq c_2(k,\beta) \left( (L+2)h^2 + \|\cW_L-\cW_{L,h}^\ell\|_{L^2(\Omega;L^2(\IS^2))}\right)
	\end{equation*}
	with
	\begin{equation*}
	c_2(k,\beta)
	= \frac{2 k |\sin(\pi \{\beta\})|}{\pi}
	c_{\gk, \{\gb\}, k}
	\max\{ c_{\gk,\floor{\gb}}^{(1)}, c_{\gk,\floor{\gb}}^{(2)}\},
	\end{equation*}
	which concludes the proof.
\end{proof}

We are now ready to state our main result on the convergence of the SFEM approximation to the solution of~\eqref{eq:theone}.

\begin{theorem}\label{th:yt1}
	Let $u$ be the solution to~\eqref{eq:theone} with $\beta>1/2$, and let $u_{L,h}^\ell$ be given by~\eqref{eq:finalapprox} for $\fbet \neq 0$ and be the lifted solution to the recursion~\eqref{eq:linproblemSFEM} in the case when $\gb$ is a positive integer.
	Then the strong error is bounded by 
	\begin{align*}
	& \|u-u_{L,h}^\ell\|_{L^2(\Omega, L^2(\IS^2))}\\
	& \qquad \le C_\gk \left(\frac{1}{2\beta-1}+\frac{1}{4\beta-1}\right)L^{1-2\beta}
	+c_1(k,L, \beta)\\
	& \hspace*{3.95em} +\!\left((1-\delta_{0 ,\fbet}) c_2(k,\beta) + \delta_{0 ,\fbet} c_3(\beta)\right)\!
	\!\left((L+2)h^2+\|\cW_L-\cW_{L,h}^\ell\|_{L^2(\Omega, L^2(\IS^2))}\right)\!
	\end{align*}
	with constants defined in Propositions~\ref{prop:truncref}, \ref{prop:quaderr}, and~\ref{lem:firststep}.
	If, in addition, $\|\cW_L-\cW_{L,h}^\ell\|_{L^2(\Omega, L^2(\IS^2))}$ is chosen as in Equation~\eqref{eq:inters}, the error to the fractional problem is for $s \in [0,2]$ bounded by
	\begin{equation*}
	\|u-u_{L,h}^\ell\|_{L^2(\Omega, L^2(\IS^2))} \leq C(L+1)\left(L^{-2\beta}+e^{-\pi^2/(4k)}+ h^s(L+1)^{s}\right).
	\end{equation*}
\end{theorem}

\begin{proof}
	Let us start with $\fbet \neq 0$. By the triangle inequality, we obtain
	\begin{align*}
	&\|u-u_{L,h}^\ell\|_{L^2(\Omega, L^2(\IS^2))}\\
	&\quad~~ \le \|u-u_L\|_{L^2(\Omega, L^2(\IS^2))} 		
	+ \|u_L-u_{L,Q,k}\|_{L^2(\Omega, L^2(\IS^2))} 	
	+ \|u_{L,Q,k} - u_{L,h}^\ell\|_{L^2(\Omega, L^2(\IS^2))},	
	\end{align*}
	and the claim follows with Propositions~\ref{prop:truncref}, \ref{prop:quaderr}, and~\ref{lem:firststep}.
	
	For $\fbet = 0$, we split
	\begin{equation*}
	\|u-u_{L,h}^\ell\|_{L^2(\Omega, L^2(\IS^2))}
	\le \|u-u_L\|_{L^2(\Omega, L^2(\IS^2))} 		
	+ \|u_L - u_{L,h}^\ell\|_{L^2(\Omega, L^2(\IS^2))}. 	
	\end{equation*}
	The first term is again bounded by Proposition~\ref{prop:truncref}, and the second term satisfies
	\begin{align*}
	\|u_L-u_{L,h}^{\ell}\|_{L^2(\Omega;L^2(\IS^2))}
	& \le h^2 \sum_{i=1}^{\beta-1} c^{\beta-i} \max \{ \gk^{1-2i}, \gk^{4\beta+1-6i} \} (2i-1)^{-1/2}\\
	&\quad +c^{\beta} \max\{1, \gk^{4\gb} \} 
	\left( h^2 (L+1) + \|\cW_L-\cW_{L,h}^\ell\|_{L^2(\Omega;L^2(\IS^2))} \right)\\
	& = c_3(\beta) \left( (L+2) h^2 + \|\cW_L-\cW_{L,h}^\ell\|_{L^2(\Omega;L^2(\IS^2))} \right),
	\end{align*}
	which is derived as in the proof of Proposition~\ref{lem:firststep}. This concludes the proof.
\end{proof}

We close this section with a short discussion of the computational complexity of the method.
	We begin by calibrating the different contributions appearing in the final error estimate in Theorem~\ref{th:yt1}. Thus we obtain for the space mesh size~$h$ with $s=2$ 
	\begin{align*}
	h \sim L^{-(\beta+1)}
	\end{align*}
	and for the quadrature step size~$k$ 
	\begin{align*}
	k \sim (\beta \log(L+1))^{-1}.
	\end{align*}
	This leads to an overall error of 
	\begin{align*}
	\|u-u_{L,h}^{\ell}\|_{L^2(\Omega;L^2(\IS^2))}
	\le C L^{1-2\beta}
	\sim h^{\frac{2\beta-1}{\beta + 1}}
	\sim e^{-k^{-1}}. 
	\end{align*}
	Then, given the expressions of $K^+$ and $K^-$, we see that the number of linear systems we have to solve is of order $\operatorname{O}(k^{-2}) + \floor{\beta}$ if $\beta \notin \N$, which in terms of~$h$ becomes a complexity of $\operatorname{O}((\ln h)^2)$. The overall complexity is essentially influenced by the choice of the solver for the linear system. Given a method to generate white noise on the finite element space, a naive conjugate gradient method to solve one linear system would need $\operatorname{O}(h^{-3})$ operations with the number of degrees of freedom assumed to behave as $\operatorname{O}(h^{-2})$. Using the sparsity of the finite element matrices reduces these costs to $\operatorname{O}(h^{-3/2})$. If we adapt the multilevel approach with a BPX-type preconditioning from~\cite{Herrmann2019} to our sequence of finite element spaces, the costs could be reduced even more to $\operatorname{O}(h^{-1}(\log h^{-1})^{3/2})$. The different approaches lead therefore to total computational costs of $\operatorname{O}(h^{-3}(\ln h)^2)$, $\operatorname{O}(h^{-3/2}(\ln h)^2)$, and $\operatorname{O}(h^{-1}|\log h|^{7/2})$, respectively.
	
	A naive Python-based implementation using FEniCS with the conjugate gradient method is available on a GitHub repository\footnote{\url{https://github.com/erik-grennberg-jansson/matern_sfem}}. See Figure~\ref{fig:fields} for examples of fields generated using this code with $L=100$. We furthermore emphasize that practitioners by no means are limited to the Python-FEniCS combination but that the method is implementable in other languages which is expected to lead to better running times. 

\section{Numerical experiment}
\label{seq:five}
Finally, we confirm the theoretical results obtained in Theorem~\ref{th:yt1} by a numerical simulation. We consider the case $\{\beta\}\neq 0$. Since the error induced by the truncation of the white noise~$\cW$ was already simulated and confirmed in~\cite{Lang2015}, we focus here on the confirmation of the quadrature and SFEM error, i.e., we want to show that
\begin{align*}
&\|u_L - u_{L,h}^\ell\|_{L^2(\gO;L^2(\IS^2))}\\
&\qquad\le c_1(k,\fbet)+ c_2(k,L,\gb) \left( (L+2)h^2 + \|\cW_L-\cW_{L,h}^\ell\|_{L^2(\Omega;L^2(\IS^2))}\right).
\end{align*}
The truncated white noise~$\cW_L$ is approximated using projection, which implies by Equation~\eqref{eq:projectionH2} that 
\begin{equation*}
\|\cW_L-\cW_{h,L}^\ell\|_{L^2(\Omega;L^2(\IS^2))} \le C (L+1)^3 h^2.
\end{equation*}
Therefore, we expect to see $h^2$ convergence for $c_1(k,L,\beta) $ sufficiently small, which is expected due to the exponential decay of $c_1(k,L,\beta) $ in~$k$.

We approximate the error by $500$~Monte Carlo samples and study first convergence with varying exponent~$\gb$ and then with varying constant~$\gk$ for fixed $L=1$. We discretize the sphere using an icosahedral uniform triangular mesh with triangle sizes $h=2^{-i}$ for $i= 1,\ldots,5$. The simulations are implemented in Python~3 using the FEniCS package~\cite{FEniCS} and performed on the local computational resources available at the Department of Mathematical Sciences at Chalmers University of Technology.

\begin{figure}[t]
	\centering
	\includegraphics[width=0.7\linewidth]{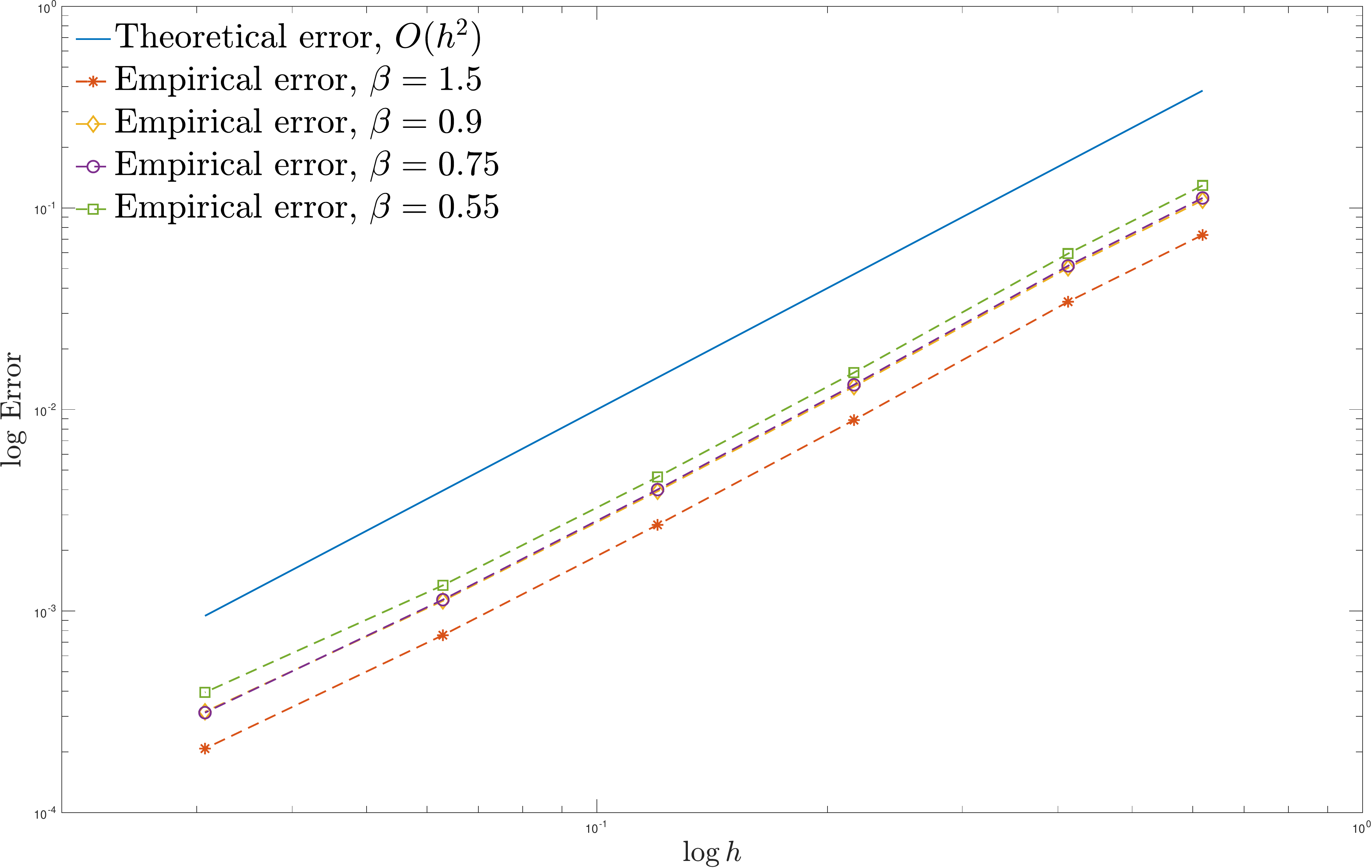}
	\caption{Strong error with varying $\beta$ and fixed $k=0.5$ and $\kappa=1$.}
	\label{fig:resvarb}
\end{figure}
In Figure~\ref{fig:resvarb} we fix $k=0.5$ and $\kappa=1$ and vary $\beta=1.5$ $\beta=0.9$, $\beta=0.75$, and $\beta=0.55$. We observe the predicted $h^2$ convergence regardless of the regularity.

Next we perform simulations for fixed $\gb = 0.75$ and varying $\gk=0.1,1,10$. We choose first $k=0.5$ as before and show the result in Figure~\ref{fig:resvarkap1}.
We observe that especially for $\gk = 0.1$, the first error term $c_1(k,L,\gb) \gk^{-1} $ seems to dominate for small~$h$. In order to decrease it, we repeat the same simulation with $k=0.1$ instead. In Figure~\ref{fig:resvarkap2} the dominance of the first error term in Figure~\ref{fig:resvarkap1} is confirmed since $h^2$ convergence is recovered now for the smaller choice of~$k$.
\begin{figure}[tbh]
	\centering
	\subfigure[$k=0.5$. \label{fig:resvarkap1}]{\includegraphics[width=0.7\linewidth]{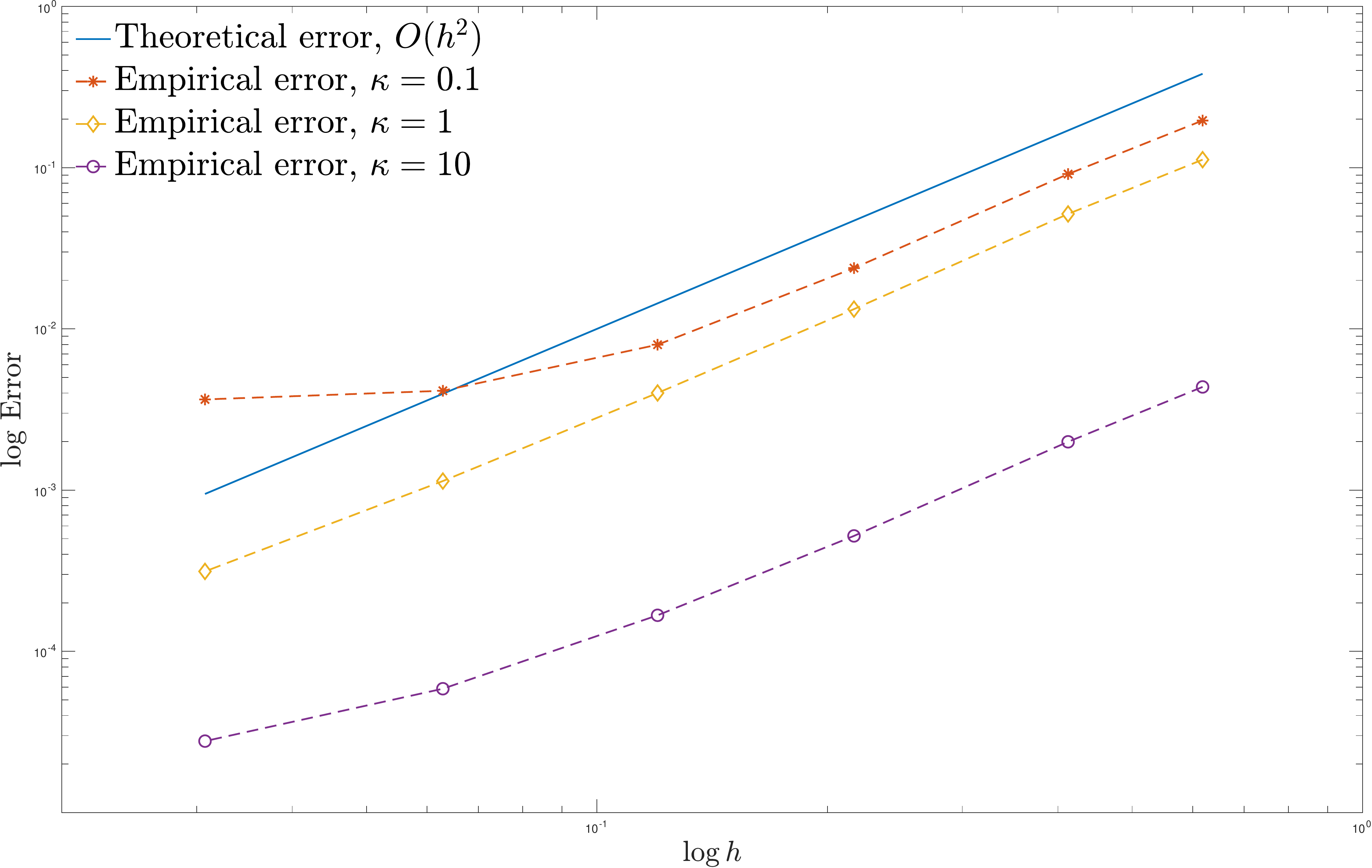}}
	\subfigure[$k=0.1$. \label{fig:resvarkap2}]{\includegraphics[width=0.7\linewidth]{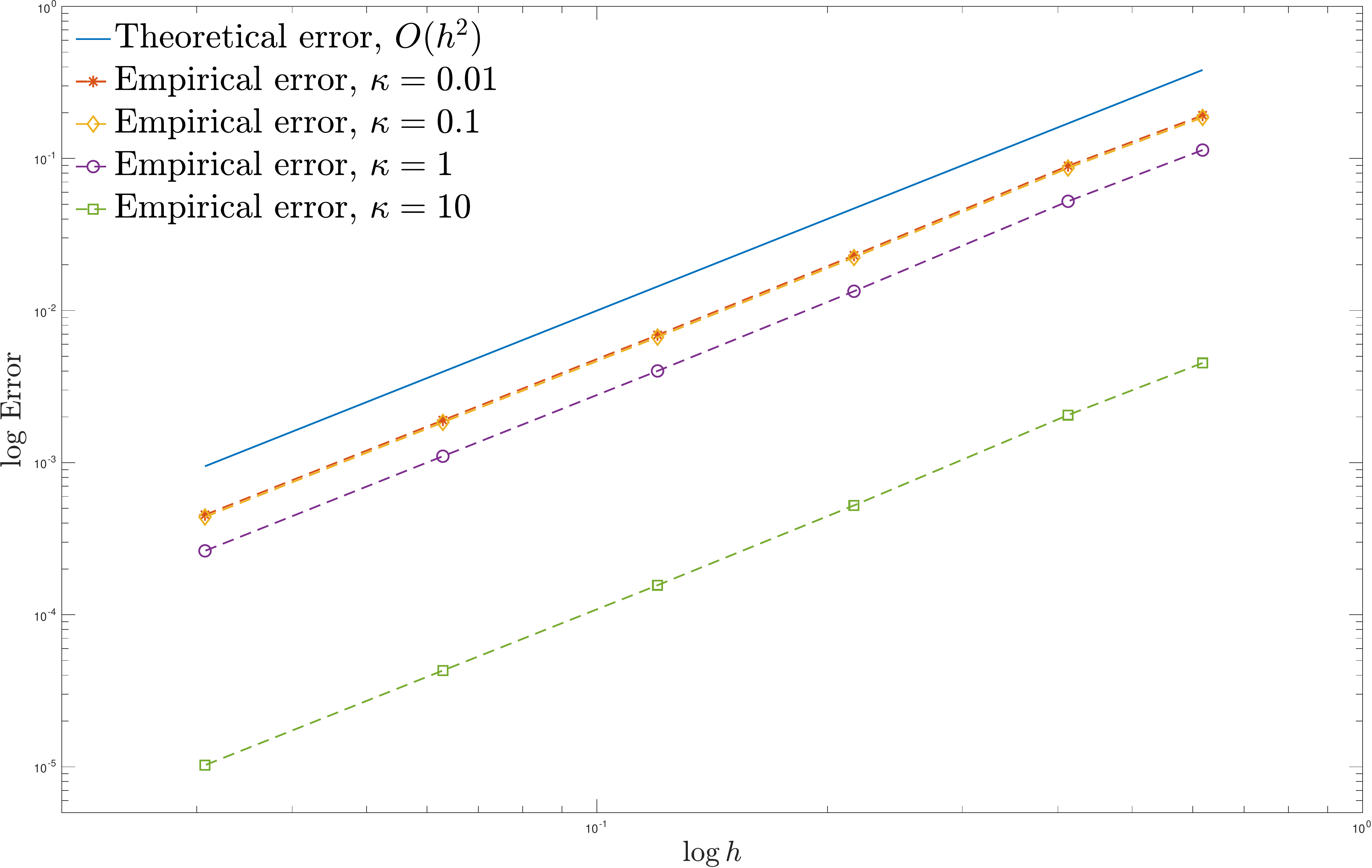}}
	\caption{Strong error with varying~$\kappa$ and $\beta = 0.75$.}
\end{figure}

		\bibliographystyle{plain}
	\bibliography{refs.bib}
	
\end{document}